\documentclass[12pt]{amsart}
\usepackage{amssymb}
\usepackage[all]{xy}

\newtheorem{theorem}{Theorem}[section]
\newtheorem{definition}[theorem]{Definition}
\newtheorem{proposition}[theorem]{Proposition}

\begin{document}

\title{Weingarten integration over noncommutative homogeneous spaces}

\author{Teodor Banica}
\address{T.B.: Department of Mathematics, Cergy-Pontoise University, 95000 Cergy-Pontoise, France. {\tt teodor.banica@u-cergy.fr}}


\begin{abstract}
We discuss an extension of the Weingarten formula, to the case of noncommutative homogeneous spaces, under suitable ``easiness'' assumptions. The spaces that we consider are noncommutative algebraic manifolds, generalizing the spaces of type $X=G/H\subset\mathbb C^N$, with $H\subset G\subset U_N$ being subgroups of the unitary group, subject to certain uniformity conditions. We discuss various axiomatization issues, then we establish the Weingarten formula, and we derive some probabilistic consequences.
\end{abstract}

\maketitle

\section*{Introduction}

Given a compact group action $G\curvearrowright X$, assumed to be transitive, we have $X=G/H$, where $H=\{g\in G|gx_0=x_0\}$ is the stabilizer of a given point $x_0\in X$. Thus, we have an embedding $C(X)\subset C(G)$. The unique $G$-invariant integration on $X$ is then obtained as a composition, $\int:C(X)\subset C(G)\to\mathbb C$, and can be explicitely computed provided that we know how to integrate over $G$, for instance via a Weingarten type formula.

We discuss here some noncommutative extensions of these facts, based on some previous work in \cite{ba1}, \cite{ba2}, \cite{bgo}, \cite{bss}. The action $O_N^+\curvearrowright S^{N-1}_{\mathbb R,+}$, which is the free analogue of the usual action $O_N\curvearrowright S^{N-1}_\mathbb R$, was studied some time ago, in \cite{bgo}. Shortly afterwards, an extension to spaces of type $G_N/G_{N-M}$, with $M\leq N$, and with $G=(G_N)$ subject to some suitable uniformity assumptions (``easiness'') was discussed in \cite{bss}. More recently, various spaces of type $(G_M\times G_N)/(G_L\times G_{M-L}\times G_{N-L})$, with $L\leq M\leq N$, and with $G=(G_N)$ belonging to more general families of quantum groups, were studied in \cite{ba1}, \cite{ba2}.

The common feature of these spaces $X=G/H$ is that they are ``easy'', in the sense that one can explicitely integrate on them, via a Weingarten type formula. The purpose of the present paper is to provide an axiomatic framework for such spaces, to advance at the level of the general theory, and to enlarge the class of known examples. 

The paper is organized as follows: 1-2 are preliminary sections, in 3-4 we restrict the attention to the affine space case, in 5-6 we discuss some basic examples, and in 7-8 we focus on the easy space case and we discuss a number of probabilistic aspects.

\medskip

\noindent {\bf Acknowledgements.} I would like to thank the referee for a careful reading of the manuscript and for a number of useful suggestions.

\section{Homogeneous spaces}

We use Woronowicz's quantum group formalism in \cite{wo1}, \cite{wo2}, with the extra assumption $S^2=id$. In other words, the quantum groups that we will consider will be the abstract duals, in the sense of the general $C^*$-algebra theory, of the Hopf $C^*$-algebras considered in \cite{wo1}, \cite{wo2}, whose antipode satisfies the usual group-theoretic  condition $S^2=id$. 

The precise definition of these latter algebras is as follows:

\begin{definition}
A finitely generated Hopf $C^*$-algebra is a unital $C^*$-algebra $A$, given with a unitary matrix $u\in M_N(A)$ whose coefficients generate $A$, such that the formulae
$$\Delta(u_{ij})=\sum_ku_{ik}\otimes u_{kj}\quad,\quad
 \varepsilon(u_{ij})=\delta_{ij}\quad,\quad
S(u_{ij})=u_{ji}^*$$
define morphisms of $C^*$-algebras $\Delta:A\to A\otimes A$, $\varepsilon:A\to\mathbb C$, $S:A\to A^{opp}$.
\end{definition}

The morphisms $\Delta,\varepsilon,S$ are called comultiplication, counit and antipode. They satisfy the usual Hopf algebra axioms, on the dense $*$-subalgebra $<u_{ij}>\subset A$.

There are two basic classes of examples of such algebras, as follows:
\begin{enumerate}
\item The function algebra $A=C(G)$, with $G\subset U_N$ being a compact Lie group, together with the matrix of standard coordinates, $u_{ij}(g)=g_{ij}$. 

\item The group algebra $A=C^*(\Gamma)$, with $\Gamma=<g_1,\ldots,g_N>$ being a finitely generated discrete group, taken with the matrix $u=diag(g_1,\ldots,g_N)$. 
\end{enumerate}

In view of these examples, we write in general $A=C(G)=C^*(\Gamma)$, with $G$ being a compact quantum group, and $\Gamma$ being a discrete quantum group. See \cite{wo1}, \cite{wo2}.

A closed quantum subgroup of a compact quantum group, $H\subset G$, corresponds by definition to a morphism of $C^*$-algebras $\rho:C(G)\to C(H)$, mapping standard coordinates to standard coordinates. Observe that such a morphism is automatically surjective, and transforms the structural maps $\Delta,\varepsilon,S$ of the algebra $C(G)$ into those of $C(H)$.

Let us recall as well that given a noncommutative compact space $X$, an action $G\curvearrowright X$ corresponds by definition to a coaction map $\Phi:C(X)\to C(G)\otimes C(X)$, which is subject to the coassociativity condition $(id\otimes\Phi)\Phi=(\Delta\otimes id)\Phi$. See e.g. \cite{bss}.

Let us discuss now the quotient space construction:

\begin{proposition}
Given a quantum subgroup $H\subset G$, with associated quotient map $\rho:C(G)\to C(H)$, if we define the quotient space $X=G/H$ by setting
$$C(X)=\left\{f\in C(G)\Big|(id\otimes\rho)\Delta f=f\otimes1\right\}$$
then we have a coaction $\Phi:C(X)\to C(G)\otimes C(X)$, obtained as the restriction of the comultiplication of $C(G)$. In the classical case, we obtain the usual space $X=G/H$.
\end{proposition}

\begin{proof}
Observe that $C(X)\subset C(G)$ is indeed a subalgebra, because it is defined via a relation of type $\varphi(f)=\psi(f)$, with $\varphi,\psi$ morphisms. Observe also that in the classical case we obtain the algebra of continuous functions on $X=G/H$, because:
\begin{eqnarray*}
(id\otimes\rho)\Delta f=f\otimes1
&\iff&(id\otimes\rho)\Delta f(g,h)=(f\otimes1)(g,h),\forall g\in G,\forall h\in H\\
&\iff&f(gh)=f(g),\forall g\in G,\forall h\in H\\
&\iff&f(gh)=f(gk),\forall g\in G,\forall h,k\in H
\end{eqnarray*}

Regarding now the construction of $\Phi$, observe that for $f\in C(X)$ we have: 
\begin{eqnarray*}
(id\otimes id\otimes\rho)(id\otimes\Delta)\Delta f
&=&(id\otimes id\otimes\rho)(\Delta\otimes id)\Delta f\\
&=&(\Delta\otimes id)(id\otimes\rho)\Delta f\\
&=&(\Delta\otimes id)(f\otimes 1)\\
&=&\Delta f\otimes1
\end{eqnarray*}

Thus $f\in C(X)$ implies $\Delta f\in C(G)\otimes C(X)$, and this gives the existence of $\Phi$. Finally, the fact that $\Phi$ is coassociative is clear from definitions, and so is the fact that, in the classical case, we obtain in this way the standard action $G\curvearrowright G/H$.
\end{proof}

As an illustration, in the group dual case we have:

\begin{proposition}
Assume that $G=\widehat{\Gamma}$ is a discrete group dual.
\begin{enumerate}
\item The quantum subgroups of $G$ are $H=\widehat{\Lambda}$, with $\Gamma\to\Lambda$ being a quotient group.

\item For such a quantum subgroup $\widehat{\Lambda}\subset\widehat{\Gamma}$, we have $\widehat{\Gamma}/\widehat{\Lambda}=\widehat{\Theta}$, where $\Theta=\ker(\Gamma\to\Lambda)$.
\end{enumerate}
\end{proposition}

\begin{proof}
The first assertion follows by using the theory from \cite{wo1}. Indeed, since the algebra $C(G)=C^*(\Gamma)$ is cocommutative, so are all its quotients, and this gives the result.

Regarding now (2), consider a quotient map $r:\Gamma\to\Lambda$, and denote by $\rho:C^*(\Gamma)\to C^*(\Lambda)$ its extension. With $f=\sum_{g\in\Gamma}\lambda_g\cdot g\in C^*(\Gamma)$ we have:
\begin{eqnarray*}
f\in C(\widehat{\Gamma}/\widehat{\Lambda})
&\iff&(id\otimes\rho)\Delta(f)=f\otimes 1\\
&\iff&\sum_{g\in\Gamma}\lambda_g\cdot g\otimes r(g)=\sum_{g\in\Gamma}\lambda_g\cdot g\otimes 1\\
&\iff&\lambda_g\cdot r(g)=\lambda_g\cdot 1,\forall g\in\Gamma\\
&\iff&supp(f)\subset\ker(r)
\end{eqnarray*}

But this means $\widehat{\Gamma}/\widehat{\Lambda}=\widehat{\Theta}$, with $\Theta=\ker(\Gamma\to\Lambda)$, as claimed.
\end{proof}

Given two noncommutative compact spaces $X,Y$, we say that $X$ is a quotient space of $Y$ when we have an embedding of $C^*$-algebras $\alpha:C(X)\subset C(Y)$. We have:

\begin{definition}
We call a quotient space $G\to X$ homogeneous when the comultiplication $\Delta:C(G)\to C(G)\otimes C(G)$ satisfies $\Delta(C(X))\subset C(G)\otimes C(X)$.
\end{definition}

In other words, an homogeneous quotient space $G\to X$ is a noncommutative space coming from a subalgebra $C(X)\subset C(G)$, which is stable under the comultiplication.

The relation with the quotient spaces from Proposition 1.2 is as follows:

\begin{theorem}
The following results hold:
\begin{enumerate}
\item The quotient spaces $X=G/H$ are homogeneous.

\item In the classical case, any homogeneous space is of type $G/H$.

\item In general, there are homogeneous spaces which are not of type $G/H$.
\end{enumerate}
\end{theorem}

\begin{proof}
Once again these results are well-known, the proof being as follows:

(1) This is clear from Proposition 1.2 above.

(2) Consider a quotient map $p:G\to X$. The invariance condition in the statement tells us that we must have an action $G\curvearrowright X$, given by $g(p(g'))=p(gg')$. Thus:
$$p(g')=p(g'')\implies p(gg')=p(gg''),\ \forall g\in G$$

Now observe that $H=\{g\in G|p(g)=p(1)\}$ is a group, because $g,h\in H$ implies $p(gh)=p(g)=p(1)$, so $gh\in H$, and the other axioms are satisfied as well. Our claim is that we have $X=G/H$, via $p(g)\to gH$. Indeed, $p(g)\to gH$ is well-defined and bijective, because $p(g)=p(g')$ is equivalent to $p(g^{-1}g')=p(1)$, so to $gH=g'H$, as desired. 

(3) Given a discrete group $\Gamma$ and an arbitrary subgroup $\Theta\subset\Gamma$, the quotient space $\widehat{\Gamma}\to\widehat{\Theta}$ is homogeneous. Now by using Proposition 1.3 above, we can see that if $\Theta\subset\Gamma$ is not normal, the quotient space $\widehat{\Gamma}\to\widehat{\Theta}$ is not of the form $G/H$.
\end{proof}

Let us try now to understand the general properties of the homogeneous spaces $G\to X$, in the sense of Theorem 1.5. We recall that any compact quantum group $G$ has a Haar integration functional $\int:C(G)\to\mathbb C$, having the following invariance properties:
$$\left(\int\otimes id\right)\Delta=\left(id\otimes\int\right)\Delta=\int(.)1$$

For the existence and uniqueness of $\int$, we refer to Woronowicz's paper \cite{wo1}. 

We have the following result, which is once again well-known:

\begin{proposition}
Assume that a quotient space $G\to X$ is homogeneous.
\begin{enumerate}
\item The restriction $\Phi:C(X)\to C(G)\otimes C(X)$ of $\Delta$ is a coaction.

\item We have $\Phi(f)=1\otimes f\implies f\in\mathbb C1$, and $(\int\otimes id)\Phi f=\int f$.

\item The restriction of $\int$ is the unique unital form satisfying $(id\otimes\tau)\Phi=\tau(.)1$.
\end{enumerate}
\end{proposition}

\begin{proof}
These results are all elementary, the proof being as follows:

(1) This is clear from definitions, because $\Delta$ itself is a coaction.

(2) If $f\in C(G)$ is such that $\Delta(f)=1\otimes f$ then $(id\otimes\varepsilon)\Delta f=(id\otimes\varepsilon)(1\otimes f)$, and so $f=\varepsilon(f)1$. Regarding the second assertion, this follows from the right invariance property $(\int\otimes id)\Delta f=\int f$ of the Haar functional of $C(G)$, by restriction to $C(X)$.

(3) The fact that $tr=\int_{|C(X)}$ is $G$-invariant, in the sense that $(id\otimes tr)\Phi f=tr(f)1$, follows from the left invariance property $(id\otimes\int)\Delta f=\int f$ of the Haar functional of $C(G)$. Conversely, assuming that $\tau:C(X)\to\mathbb C$ satisfies $(id\otimes\tau)\Phi f=\tau(f)1$, we have:
$$\left(\int\otimes\tau\right)\Phi(f)
=\begin{cases}
\int(id\otimes\tau)\Phi(f)=\int(\tau(f)1)=\tau(f)\\
\tau (\int\otimes id)\Phi(f)=\tau(tr(f)1)=tr(f)
\end{cases}$$

Thus we have $\tau(f)=tr(f)$ for any $f\in C(X)$, and this finishes the proof.
\end{proof}

Summarizing, we have a notion of noncommutative homogeneous space, which perfectly covers the classical case. In general, however, the group dual case shows that our formalism is more general than that of the quotient spaces $G/H$. See \cite{boc}, \cite{dya}, \cite{kso}, \cite{pod}.

\section{Extended formalism}

We discuss now an extra issue, of analytic nature. The point is that for one of the most basic examples of actions, $O_N^+\curvearrowright S^{N-1}_{\mathbb R,+}$, the associated morphism $\alpha:C(X)\to C(G)$ is not injective. In order to include such examples, we must relax our axioms:

\begin{definition}
An extended homogeneous space consists of a morphism of $C^*$-algebras $\alpha:C(X)\to C(G)$, and a coaction map $\Phi:C(X)\to C(G)\otimes C(X)$, such that
$$\xymatrix@R=12mm@C=20mm{
C(X)\ar[r]^\Phi\ar[d]_\alpha&C(G)\otimes C(X)\ar[d]^{id\otimes\alpha}\\
C(G)\ar[r]^\Delta&C(G)\otimes C(G)
}\qquad\qquad
\xymatrix@R=12mm@C=20mm{
C(X)\ar[r]^\Phi\ar[d]_\alpha&C(G)\otimes C(X)\ar[d]^{\int\otimes id}\\
C(G)\ar[r]^{\int(.)1}&C(X)
}$$
both commute, where $\int$ is the Haar integration over $G$. We then write $G\to X$.
\end{definition}

When $\alpha$ is injective we obtain an homogeneous space in the sense of section 1. The examples with $\alpha$ not injective include the standard action $O_N^+\curvearrowright S^{N-1}_{\mathbb R,+}$, for which we refer to \cite{bgo}, and the standard action $U_N^+\curvearrowright S^{N-1}_{\mathbb C,+}$, discussed in section 3 below.

Here are a few general remarks on the above axioms:

\begin{proposition}
Assume that we have morphisms of $C^*$-algebras $\alpha:C(X)\to C(G)$ and $\Phi:C(X)\to C(G)\otimes C(X)$, satisfying $(id\otimes\alpha)\Phi=\Delta\alpha$.
\begin{enumerate}
\item If $\alpha$ is injective on a dense $*$-subalgebra $A\subset C(X)$, and $\Phi(A)\subset C(G)\otimes A$, then $\Phi$ is automatically a coaction map, and is unique.

\item The ergodicity type condition $(\int\otimes id)\Phi=\int\alpha(.)1$ is equivalent to the existence of a linear form $\lambda:C(X)\to\mathbb C$ such that $(\int\otimes id)\Phi=\lambda(.)1$.
\end{enumerate}
\end{proposition}

\begin{proof}
Assuming that we have a dense $*$-subalgebra $A\subset C(X)$ as in (1), the restriction $\Phi_{|A}$ is given by $\Phi_{|A}=(id\otimes\alpha_{|A})^{-1}\Delta\alpha_{|A}$, and is therefore coassociative, and unique. By continuity, $\Phi$ itself follows to be coassociative and unique.

Regarding now (2), assuming $(\int\otimes id)\Phi=\lambda(.)1$, we have $(\int\otimes\alpha)\Phi=\lambda(.)1$. But $(\int\otimes\alpha)\Phi=(\int\otimes id)\Delta\alpha=\int\alpha(.)1$, and so we have $\lambda=\int\alpha$, as claimed.
\end{proof}

Given an extended homogeneous space $G\to X$, with associated map $\alpha:C(X)\to C(G)$, we can consider the image of this latter map, $\alpha:C(X)\to C(Y)\subset C(G)$. Equivalently, at the level of noncommutative spaces, we can factorize $G\to Y\subset X$. We have:

\begin{proposition}
Consider an extended homogeneous space $G\to X$.
\begin{enumerate}
\item $\Phi(f)=1\otimes f\implies f\in\mathbb C1$.

\item $tr=\int\alpha$ is the unique unital $G$-invariant form on $C(X)$.

\item The image space obtained by factorizing, $G\to Y$, is homogeneous.
\end{enumerate}
\end{proposition}

\begin{proof}
The first assertion follows from $(\int\otimes id)\Phi(f)=\int\alpha(f)1$, which gives $f=\int\alpha(f)1$. The fact that $tr=\int\alpha$ is indeed $G$-invariant can be checked as follows:
$$(id\otimes tr)\Phi f=(id\otimes\smallint\alpha)\Phi f=(id\otimes\smallint)\Delta\alpha f=\smallint\alpha(f)1=tr(f)1$$

As for the uniqueness assertion, this follows as in the proof of Proposition 1.6.

Finally, the condition $(id\otimes\alpha)\Phi=\Delta\alpha$, together with the fact that $i$ is injective, allows us to factorize $\Delta$ into a morphism $\Psi$, as follows:
$$\xymatrix@R=10mm@C=30mm{
C(X)\ar[r]^\Phi\ar[d]_\alpha&C(G)\otimes C(X)\ar[d]^{id\otimes\alpha}\\
C(Y)\ar@.[r]^\Psi\ar[d]_i&C(G)\otimes C(Y)\ar[d]^{id\otimes i}\\
C(G)\ar[r]^\Delta&C(G)\otimes C(G)
}$$

Thus the image space $G\to Y$ is indeed homogeneous, and we are done.
\end{proof}

Finally, we have the following result:

\begin{theorem}
Let $G\to X$ be an extended homogeneous space, and construct quotients $X\to X'$, $G\to G'$ by performing the GNS construction with respect to $\int\alpha,\int$. Then $\alpha$ factorizes into an inclusion $\alpha':C(X')\to C(G')$, and we have an homogeneous space.
\end{theorem}

\begin{proof}
We factorize $G\to Y\subset X$ as in Proposition 2.3 (3). By performing the GNS construction with respect to $\int i\alpha,\int i,\int$, we obtain a diagram as follows:
$$\xymatrix@R=10mm@C=30mm{
C(X)\ar[r]^p\ar[d]_\alpha&C(X')\ar[d]^{\alpha'}\ar[dr]^{tr'}\\
C(Y)\ar[r]^q\ar[d]_i&C(Y')\ar[d]^{i'}&\mathbb C\\
C(G)\ar[r]^r&C(G')\ar[ur]_{\int'}
}$$

Indeed, with $tr=\int\alpha$, the GNS quotient maps $p,q,r$ are defined respectively by:
\begin{eqnarray*}
\ker p&=&\{f\in C(X)|tr(f^*f)=0\}\\
\ker q&=&\{f\in C(Y)|\smallint(f^*f)=0\}\\
\ker r&=&\{f\in C(G)|\smallint(f^*f)=0\}
\end{eqnarray*}

Next, we can define factorizations $i',\alpha'$ as above. Observe that $i'$ is injective, and that $\alpha'$ is surjective. Our claim now is that $\alpha'$ is injective as well. Indeed:
$$\alpha'p(f)=0\implies q\alpha(f)=0\implies\int\alpha(f^*f)=0\implies tr(f^*f)=0\implies p(f)=0$$

We conclude that we have $X'=Y'$, and this gives the result.
\end{proof}

\section{Affine spaces}

We discuss now the case that we are really interested in, where $X$ is an algebraic manifold, and $G$ acts affinely on it. Let us first recall that the free complex sphere $S^{N-1}_{\mathbb C,+}$ and the free unitary quantum group $U_N^+$ are constructed as follows:
\begin{eqnarray*}
C(S^{N-1}_{\mathbb C,+})&=&C^*\left(x_1,\ldots,x_N\Big|\sum_ix_ix_i^*=\sum_ix_i^*x_i=1\right)\\
C(U_N^+)&=&C^*\left((u_{ij})_{i,j=1,\ldots,N}\Big|u^*=u^{-1},u^t=\bar{u}^{-1}\right)
\end{eqnarray*}

Here $u=(u_{ij})$ is the square matrix formed by the generators of $C(U_N^+)$. See \cite{wan}.

It is known that $S^{N-1}_{\mathbb C,+}$ is an extended homogeneous space over $U_N^+$, the associated morphisms $\alpha,\Phi$ being given by $\alpha(x_i)=u_{i1}$ and $\Phi(x_i)=\sum_au_{ia}\otimes x_a$. See \cite{bgo}.

Motivated by this fundamental example, let us formulate:

\begin{definition}
An extended homogeneous space $G\to X$ is called affine when $X\subset S^{N-1}_{\mathbb C,+}$ is an algebraic submanifold, $G\subset U_N^+$ is a closed quantum subgroup, and we have
$$\alpha(x_i)=\frac{1}{\sqrt{|I|}}\sum_{b\in I}u_{ib}\quad,\quad
\Phi(x_i)=\sum_au_{ia}\otimes x_a$$
for a certain set of indices $I\subset\{1,\ldots,N\}$. 
\end{definition}

Here the notion of algebraic manifold is the usual one, the coordinates $x_1,\ldots,x_N$ being subject to a number of (noncommutative) polynomial relations. As for the notion of quantum subgroup, we use here the general formalism from section 1 above.

Observe that $U_N^+\to S^{N-1}_{\mathbb C,+}$ is indeed affine in this sense, with $I=\{1\}$. Observe also that the $1/\sqrt{|I|}$ constant appearing above is the correct one, because: 
$$\sum_i\left(\sum_{b\in I}u_{ib}\right)\left(\sum_{b\in I}u_{ib}\right)^*=\sum_i\sum_{b,c\in I}u_{ib}u_{ic}^*=\sum_{b,c\in I}(u^t\bar{u})_{bc}=|I|$$

In general now, a first remark is that the first extended homogeneous space axiom in Definition 2.1, namely $(id\otimes\alpha)\Phi=\Delta\alpha$, is automatic, because we have:
\begin{eqnarray*}
(id\otimes\alpha)\Phi(x_i)&=&\sum_au_{ia}\otimes\alpha(x_a)=\frac{1}{\sqrt{|I|}}\sum_a\sum_{b\in I}u_{ia}\otimes u_{ab}\\
\Delta\alpha(x_i)&=&\frac{1}{\sqrt{|I|}}\sum_{b\in I}\Delta(u_{ib})=\frac{1}{\sqrt{|I|}}\sum_{b\in I}\sum_au_{ia}\otimes u_{ab}
\end{eqnarray*}

We make the standard convention that all the tensor exponents $k$ are ``colored integers'', that is, $k=e_1\ldots e_k$ with $e_i\in\{\circ,\bullet\}$, with $\circ$ corresponding to the usual variables, and with $\bullet$ corresponding to their adjoints. With this convention, we have:

\begin{proposition}
The ergodicity condition $(\int\otimes id)\Phi=\int\alpha(.)1$ is equivalent to
$$(Px^{\otimes k})_{i_1\ldots i_k}=\frac{1}{\sqrt{|I|^k}}\sum_{b_1\ldots b_k\in I}P_{i_1\ldots i_k,b_1\ldots b_k}\quad,\quad\forall k,\forall i_1,\ldots,i_k$$
where $P_{i_1\ldots i_k,j_1\ldots j_k}=\int u_{i_1j_1}^{e_1}\ldots u_{i_kj_k}^{e_k}$, and where $(x^{\otimes k})_{i_1\ldots i_k}=x_{i_1}^{e_1}\ldots x_{i_k}^{e_k}$.
\end{proposition}

\begin{proof}
We have indeed the following computation:
\begin{eqnarray*}
&&\left(\int\otimes id\right)\Phi=\int\alpha(.)1\\
&\iff&\left(\int\otimes id\right)\Phi(x_{i_1}^{e_1}\ldots x_{i_k}^{e_k})=\int\alpha(x_{i_1}^{e_1}\ldots x_{i_k}^{e_k}),\forall k,\forall i_1,\ldots i_k\\
&\iff&\sum_{a_1\ldots a_k}P_{i_1\ldots i_k,a_1\ldots a_k}x_{a_1}^{e_1}\ldots x_{a_k}^{e_k}=\frac{1}{\sqrt{|I|^k}}\sum_{b_1\ldots b_k\in I}P_{i_1\ldots i_k,b_1\ldots b_k},\forall k,\forall i_1,\ldots,i_k
\end{eqnarray*}

But this gives the formula in the statement, and we are done.
\end{proof}

As a consequence, we have the following result:

\begin{theorem}
Given a closed quantum subgroup $G\subset U_N^+$, and a set $I\subset\{1,\ldots,N\}$, if we consider the following $C^*$-subalgebra and the following quotient $C^*$-algebra,
\begin{eqnarray*}
C(X_{G,I}^{min})&=&\left<\frac{1}{\sqrt{|I|}}\sum_{b\in I}u_{ib}\Big|i=1,\ldots,N\right>\subset C(G)\\
C(X_{G,I}^{max})&=&C(S^{N-1}_{\mathbb C,+})\Big/\left<(Px^{\otimes k})_{i_1\ldots i_k}=\frac{1}{\sqrt{|I|^k}}\sum_{b_1\ldots b_k\in I}P_{i_1\ldots i_k,b_1\ldots b_k}\Big|\forall k,\forall i_1,\ldots i_k\right>
\end{eqnarray*}
then we have maps $G\to X_{G,I}^{min}\subset X_{G,I}^{max}\subset S^{N-1}_{\mathbb C,+}$, the space $G\to X_{G,I}^{max}$ is affine extended homogeneous, and any affine homogeneous space $G\to X$ appears as $X_{G,I}^{min}\subset X\subset X_{G,I}^{max}$.
\end{theorem}

\begin{proof}
Consider the standard coordinates on $X_{G,I}^{min}$, namely $X_i=\frac{1}{\sqrt{|I|}}\sum_{b\in I}u_{ib}$. The fact that we have $X_{G,I}^{min}\subset S^{N-1}_{\mathbb C,+}$ follows from the following computations:
\begin{eqnarray*}
\sum_iX_iX_i^*&=&\frac{1}{|I|}\sum_i\sum_{b,c\in I}u_{ib}u_{ic}^*=\frac{1}{|I|}\sum_{b,c\in I}(u^t\bar{u})_{bc}=1\\
\sum_iX_i^*X_i&=&\frac{1}{|I|}\sum_i\sum_{b,c\in I}u_{ib}^*u_{ic}=\frac{1}{|I|}\sum_{b,c\in I}(u^*u)_{bc}=1
\end{eqnarray*}

In order to prove now that we have $X_{G,I}^{min}\subset X_{G,I}^{max}$, we must check the fact that the defining relations for $X_{G,I}^{max}$ are satisfied by the variables $X_i$. But, we have indeed:
\begin{eqnarray*}
(PX^{\otimes k})_{i_1\ldots i_k}
&=&\frac{1}{\sqrt{|I|^k}}\sum_{a_1\ldots a_k}P_{i_1\ldots i_k,a_1\ldots a_k}\sum_{b_1\ldots b_k\in I}u_{a_1b_1}^{e_1}\ldots u_{a_kb_k}^{e_k}\\
&=&\frac{1}{\sqrt{|I|^k}}\sum_{b_1\ldots b_k\in I}(Pu^{\otimes k})_{i_1\ldots i_k,b_1\ldots b_k}\\
&=&\frac{1}{\sqrt{|I|^k}}\sum_{b_1\ldots b_k\in I}P_{i_1\ldots i_k,b_1\ldots b_k}
\end{eqnarray*}

Here $Pu^{\otimes k}=P$ comes from the invariance properties of $\int$. See \cite{wo1}.

Let us prove now that we have an action $G\curvearrowright X_{G,I}^{max}$. For this purpose, we must show that the variables $Z_i=\sum_au_{ia}\otimes x_a$ satisfy the defining relations for $X_{G,I}^{max}$. We have:
\begin{eqnarray*}
(PZ^{\otimes k})_{i_1\ldots i_k}
&=&\sum_{a_1\ldots a_k}P_{i_1\ldots i_k,a_1\ldots a_k}\sum_{c_1\ldots c_k}u_{a_1c_1}^{e_1}\ldots u_{a_kc_k}^{e_k}\otimes x_{c_1}^{e_1}\ldots x_{c_k}^{e_k}\\
&=&\sum_{c_1\ldots c_k}(Pu^{\otimes k})_{i_1\ldots i_k,c_1\ldots c_k}\otimes x_{c_1}^{e_1}\ldots x_{c_k}^{e_k}=\sum_{c_1\ldots c_k}P_{i_1\ldots i_k,c_1\ldots c_k}\otimes x_{c_1}^{e_1}\ldots x_{c_k}^{e_k}\\
&=&1\otimes \frac{1}{\sqrt{|I|^k}}(Px^{\otimes k})_{i_1\ldots i_k}
=1\otimes\frac{1}{\sqrt{|I|^k}}\sum_{b_1\ldots b_k\in I}P_{i_1\ldots i_k,b_1\ldots b_k}
\end{eqnarray*}

Thus we have an action $G\curvearrowright X_{G,I}^{max}$, and since this action is ergodic by Proposition 3.2, we have an extended homogeneous space. Finally, the last assertion is clear.
\end{proof}

As a conclusion, the affine homogeneous spaces over a given closed subgroup $G\subset U_N^+$, in the sense of Definition 3.1, are the intermediate spaces $X_{G,I}^{min}\subset X\subset X_{G,I}^{max}$ having an action of $G$, with the maximal space $X_{G,I}^{max}$ known to be affine homogeneous.

\section{Integration theory}

In this section we improve Theorem 3.3, by constructing a ``canonical'' intermediate space $X_{G,I}^{min}\subset X_{G,I}\subset X_{G,I}^{max}$, using the Schur-Weyl dual of $G$, and we present as well a Weingarten integration formula, valid for any affine homogeneous space $G\to X$.

Let us first recall the usual Weingarten formula \cite{bco}, \cite{csn}, \cite{wei}:

\begin{proposition}
Assuming that $\{\xi_\pi|\pi\in D\}$ is a basis of $Fix(u^{\otimes k})$, we have
$$\int u_{i_1j_1}^{e_1}\ldots u_{i_kj_k}^{e_k}=\sum_{\pi,\sigma\in D}(\xi_\pi)_{i_1\ldots i_k}\overline{(\xi_\sigma)_{j_1\ldots j_k}}W_{kN}(\pi,\sigma)$$
where $W_{kN}=G_{kN}^{-1}$, with $G_{kN}(\pi,\sigma)=<\xi_\pi,\xi_\sigma>$.
\end{proposition}

\begin{proof}
When the exponent $k=e_1\ldots e_k$ is fixed, and the indices $i_1,\ldots,i_k$ and $j_1,\ldots,j_k$ vary, the quantities on the left in the statement form the matrix $P$, and the quantities on the right form a certain matrix $P'$. We must prove that we have $P=P'$. 

For any vector $x\in(\mathbb C^N)^{\otimes k}$, written $x=(x_{i_1\ldots i_k})$, we have:
\begin{eqnarray*}
(P'x)_{i_1\ldots i_k}
&=&\sum_{j_1\ldots j_k}\sum_{\pi,\sigma\in D}(\xi_\pi)_{i_1\ldots i_k}\overline{(\xi_\sigma)_{j_1\ldots j_k}}W_{kN}(\pi,\sigma)x_{j_1\ldots j_k}\\
&=&\sum_{\pi,\sigma\in D}<x,\xi_\sigma>W_{kN}(\pi,\sigma)(\xi_\pi)_{i_1\ldots i_k}
\end{eqnarray*}

Since this equality holds for any choice of $i_1,\ldots,i_k$, we deduce that we have:
$$P'x=\sum_{\pi,\sigma\in D}<x,\xi_\sigma>W_{kN}(\pi,\sigma)\xi_\pi$$

By standard linear algebra, we have then $Px=P'x$, and so $P=P'$. See \cite{bco}.
\end{proof}

As a first application, we have the following result:

\begin{proposition}
If $G\to X$ is an extended homogeneous space, the integration map $\int_X=\int\alpha$ is given by the Weingarten type formula
$$\int_Xx_{i_1}^{e_1}\ldots x_{i_k}^{e_k}=\sum_{\pi,\sigma\in D}(\xi_\pi)_{i_1\ldots i_k}K_I(\sigma)W_{kN}(\pi,\sigma)$$
where $\{\xi_\pi|\pi\in D\}$ is a basis of $Fix(u^{\otimes k})$, and $K_I(\sigma)=\frac{1}{\sqrt{|I|^k}}\sum_{b_1\ldots b_k\in I}\overline{(\xi_\sigma)_{b_1\ldots b_k}}$.
\end{proposition}

\begin{proof}
By using the formula in Proposition 4.1, we have:
\begin{eqnarray*}
\int_Xx_{i_1}^{e_1}\ldots x_{i_k}^{e_k}
&=&\frac{1}{\sqrt{|I|^k}}\sum_{b_1\ldots b_k\in I}\int u_{i_1b_1}^{e_1}\ldots u_{i_kb_k}^{e_k}\\
&=&\frac{1}{\sqrt{|I|^k}}\sum_{b_1\ldots b_k\in I}\sum_{\pi,\sigma\in D}(\xi_\pi)_{i_1\ldots i_k}\overline{(\xi_\sigma)_{b_1\ldots b_k}}W_{kN}(\pi,\sigma)
\end{eqnarray*}

But this gives the formula in the statement, and we are done.
\end{proof}

Let us go back now to Theorem 3.3. We know from there that $X_{G,I}^{max}\subset S^{N-1}_{\mathbb C,+}$ is constructed by imposing to the coordinates the conditions $Px^{\otimes k}=P^I$, where:
$$P_{i_1\ldots i_k,j_1\ldots j_k}=\int u_{i_1j_1}^{e_1}\ldots u_{i_kj_k}^{e_k}\quad,\quad
P^I_{i_1\ldots i_k}=\frac{1}{\sqrt{|I|^k}}\sum_{j_1\ldots j_k\in I}P_{i_1\ldots i_k,j_1\ldots j_k}$$

These quantities can be computed by using the Weingarten formula, and working out the details leads to the construction of a certain smaller space $X_{G,I}$, as follows:

\begin{theorem}
Given a closed quantum subgroup $G\subset U_N^+$, and a set $I\subset\{1,\ldots,N\}$, if we consider the following quotient algebra
$$C(X_{G,I})=C(S^{N-1}_{\mathbb C,+})\Big/\left<\sum_{a_1\ldots a_k}\xi_{a_1\ldots a_k}x_{a_1}^{e_1}\ldots x_{a_k}^{e_k}=\frac{1}{\sqrt{|I|^k}}\sum_{b_1\ldots b_k\in I}\xi_{b_1\ldots b_k}\Big|\forall k,\forall\xi\in Fix(u^{\otimes k})\right>$$
we obtain in this way an affine homogeneous space $G\to X_{G,I}$.
\end{theorem}

\begin{proof}
We use Theorem 3.3. Let us first prove that we have an inclusion $X_{G,I}\subset X_{G,I}^{max}$. According to the integration formula in Proposition 4.1, we have:
\begin{eqnarray*}
(Px^{\otimes k})_{i_1\ldots i_k}&=&\sum_{a_1\ldots a_k}\sum_{\pi,\sigma\in D}(\xi_\pi)_{i_1\ldots i_k}\overline{(\xi_\sigma)_{a_1\ldots a_k}}W_{kN}(\pi,\sigma)x_{a_1}^{e_1}\ldots x_{a_k}^{e_k}\\
P^I_{i_1\ldots i_k}&=&\frac{1}{\sqrt{|I|^k}}\sum_{b_1\ldots b_k\in I}\sum_{\pi,\sigma\in D}(\xi_\pi)_{i_1\ldots i_k}\overline{(\xi_\sigma)_{b_1\ldots b_k}}W_{kN}(\pi,\sigma)
\end{eqnarray*}

We can see that the defining relations for $X_{G,I}\subset S^{N-1}_{\mathbb C,+}$ imply $Px^{\otimes k}=P^I$, and so imply the relations defining $X_{G,I}^{max}\subset S^{N-1}_{\mathbb C,+}$. Thus, we have an inclusion $X_{G,I}\subset X_{G,I}^{max}$.

Let us prove now that we have $X_{G,I}^{min}\subset X_{G,I}$. We must check here that the variables $X_i=\frac{1}{\sqrt{|I|}}\sum_{b\in I}u_{ib}\in C(X_{G,I}^{min})$ satisfy the relations defining $X_{G,I}$, and we have indeed:
\begin{eqnarray*}
\sum_{a_1\ldots a_k}\xi_{a_1\ldots a_k}X_{a_1}^{e_1}\ldots X_{a_k}^{e_k}
&=&\frac{1}{\sqrt{|I|^k}}\sum_{a_1\ldots a_k}\xi_{a_1\ldots a_k}\sum_{b_1\ldots b_k\in I}u_{a_1b_1}^{e_1}\ldots u_{a_kb_k}^{e_k}\\
&=&\frac{1}{\sqrt{|I|^k}}\sum_{b_1\ldots b_k\in I}\xi_{b_1\ldots b_k}
\end{eqnarray*}

Finally, in order to construct an action $G\curvearrowright X_{G,I}$, we must show that the variables $Z_a=\sum_iu_{ai}\otimes x_i$ satisfy the defining relations for $X_{G,I}$. We have:
\begin{eqnarray*}
\sum_{a_1\ldots a_k}\xi_{a_1\ldots a_k}Z_{a_1}^{e_1}\ldots Z_{a_k}^{e_k}
&=&\sum_{a_1\ldots a_k}\sum_{i_1\ldots i_k}\xi_{a_1\ldots a_k}u_{a_1i_1}^{e_1}\ldots u_{a_ki_k}^{e_k}\otimes x_{i_1}^{e_1}\ldots x_{i_k}^{e_k}\\
&=&\sum_{i_1\ldots i_k}\xi_{i_1\ldots i_k}\otimes x_{i_1}^{e_1}\ldots x_{i_k}^{e_k}=1\otimes\frac{1}{\sqrt{|I|^k}}\sum_{b_1\ldots b_k\in I}\xi_{b_1\ldots b_k}
\end{eqnarray*}

Thus we have an action $G\curvearrowright X_{G,I}$, and this finishes the proof.
\end{proof}

\section{Basic examples}

We discuss now some basic examples of affine homogeneous spaces, namely those coming from the classical groups, and those coming from the group duals. We will need: 

\begin{proposition}
Assuming that a closed subset $X\subset S^{N-1}_{\mathbb C,+}$ is affine homogeneous over a classical group, $G\subset U_N$, then $X$ itself must be classical, $X\subset S^{N-1}_\mathbb C$.
\end{proposition}

\begin{proof}
We use the well-known fact that, since the standard coordinates $u_{ij}\in C(G)$ commute, the corepresentation $u^{\circ\circ\bullet\bullet}=u^{\otimes 2}\otimes\bar{u}^{\otimes 2}$ has the following fixed vector:
$$\xi=\sum_{ij}e_i\otimes e_j\otimes e_i\otimes e_j$$

With $k=\circ\circ\bullet\,\bullet$ and with this vector $\xi$, the formula in Theorem 4.3 reads:
$$\sum_{ij}x_ix_jx_i^*x_j^*=\frac{1}{\sqrt{|I|^4}}\sum_{i,j\in I}1=1$$

By using this formula, along with $\sum_ix_ix_i^*=\sum_ix_i^*x_i=1$, we obtain:
\begin{eqnarray*}
\sum_{ij}(x_ix_j-x_jx_i)(x_j^*x_i^*-x_i^*x_j^*)
&=&\sum_{ij}x_ix_jx_j^*x_i^*-x_ix_jx_i^*x_j^*-x_jx_ix_j^*x_i^*+x_jx_ix_i^*x_j^*\\
&=&1-1-1+1=0
\end{eqnarray*}

We conclude that we have $[x_i,x_j]=0$, for any $i,j$. By using now this commutation relation, plus once again the relations defining $S^{N-1}_{\mathbb C,+}$, we have as well:
\begin{eqnarray*}
\sum_{ij}(x_ix_j^*-x_j^*x_i)(x_jx_i^*-x_i^*x_j)
&=&\sum_{ij}x_ix_j^*x_jx_i^*-x_ix_j^*x_i^*x_j-x_j^*x_ix_jx_i^*+x_j^*x_ix_i^*x_j\\
&=&\sum_{ij}x_ix_j^*x_jx_i^*-x_ix_i^*x_j^*x_j-x_j^*x_jx_ix_i^*+x_j^*x_ix_i^*x_j\\
&=&1-1-1+1=0
\end{eqnarray*}

Thus we have $[x_i,x_j^*]=0$ as well, and so $X\subset S^{N-1}_\mathbb C$, as claimed. 
\end{proof}

We can now formulate the result in the classical case, as follows:

\begin{proposition}
In the classical case, $G\subset U_N$, there is only one affine homogeneous space, for each index set $I=\{1,\ldots,N\}$, namely the quotient space 
$$X=G/(G\cap C_N^I)$$
where $C_N^I\subset U_N$ is the group of unitaries fixing the vector $\xi_I=\frac{1}{\sqrt{|I|}}(\delta_{i\in I})_i$.
\end{proposition}

\begin{proof}
Consider an affine homogeneous space $G\to X$. We already know from Proposition 5.1 above that $X$ is classical. We will first prove that we have $X=X_{G,I}^{min}$, and then we will prove that $X_{G,I}^{min}$ equals the quotient space in the statement.

(1) We use the well-known fact that the functional $E=(\int\otimes id)\Phi$ is the projection onto the fixed point algebra $C(X)^\Phi=\{f\in C(X)|\Phi(f)=1\otimes f\}$. Thus our ergodicity condition, namely $E=\int\alpha(.)1$, shows that we must have $C(X)^\Phi=\mathbb C1$. Now since in the classical case the condition $\Phi(f)=1\otimes f$ reads $f(gx)=f(x)$ for any $g\in G$ and $x\in X$, we recover in this way the usual ergodicity condition, stating that whenever a function $f\in C(X)$ is constant on the orbits of the action, it must be constant. 

Now observe that for an affine action, the orbits are closed. Thus an affine action which is ergodic must be transitive, and we deduce from this that we have $X=X_{G,I}^{min}$. 

(2) We know that the inclusion $C(X)\subset C(G)$ comes via $x_i=\frac{1}{\sqrt{|I|}}\sum_{j\in I}u_{ij}$. Thus, the quotient map $p:G\to X\subset S^{N-1}_\mathbb C$ is given by the following formula:
$$p(g)=\left(\frac{1}{\sqrt{|I|}}\sum_{j\in I}g_{ij}\right)_i$$

In particular, the image of the unit matrix $1\in G$ is the following vector:
$$p(1)=\left(\frac{1}{\sqrt{|I|}}\sum_{j\in I}\delta_{ij}\right)_i=\left(\frac{1}{\sqrt{|I|}}\delta_{i\in I}\right)_i=\xi_I$$

But this gives the result, and we are done.
\end{proof}

Let us discuss now the group dual case. Given a discrete group $\Gamma=<g_1,\ldots,g_N>$, we can consider the embedding $\widehat{\Gamma}\subset U_N^+$ given by $u_{ij}=\delta_{ij}g_i$. We have then:

\begin{proposition}
In the group dual case, $G=\widehat{\Gamma}$ with $\Gamma=<g_1,\ldots,g_N>$, we have
$$X=\widehat{\Gamma}_I\quad,\quad\Gamma_I=<g_i|i\in I>\subset\Gamma$$
for any affine homogeneous space $X$, when identifying full and reduced group algebras.
\end{proposition}

\begin{proof}
Assume indeed that we have an affine homogeneous space $G\to X$, as in Definition 3.1. In terms of the rescaled coordinates $h_i=\sqrt{|I|}x_i$, our axioms for $\alpha,\Phi$ read:
$$\alpha(h_i)=\delta_{i\in I}g_i\quad,\quad\Phi(h_i)=g_i\otimes h_i$$

As for the ergodicity condition, this translates as follows:
\begin{eqnarray*}
&&\left(\int\otimes id\right)\Phi(h_{i_1}^{e_1}\ldots h_{i_p}^{e_p})=\int\alpha(h_{i_1}^{e_p}\ldots h_{i_p}^{e_p})\\
&\iff&\left(\int\otimes id\right)(g_{i_1}^{e_1}\ldots g_{i_p}^{e_p}\otimes h_{i_1}^{e_1}\ldots h_{i_p}^{e_p})=\int_G\delta_{i_1\in I}\ldots\delta_{i_p\in I}g_{i_1}^{e_1}\ldots g_{i_p}^{e_p}\\
&\iff&\delta_{g_{i_1}^{e_1}\ldots g_{i_p}^{e_p},1}h_{i_1}^{e_1}\ldots h_{i_p}^{e_p}=\delta_{g_{i_1}^{e_1}\ldots g_{i_p}^{e_p},1}\delta_{i_1\in I}\ldots\delta_{i_p\in I}\\
&\iff&\left[g_{i_1}^{e_1}\ldots g_{i_p}^{e_p}=1\implies h_{i_1}^{e_1}\ldots h_{i_p}^{e_p}=\delta_{i_1\in I}\ldots\delta_{i_p\in I}\right]
\end{eqnarray*}

Now observe that from $g_ig_i^*=g_i^*g_i=1$ we obtain in this way $h_ih_i^*=h_i^*h_i=\delta_{i\in I}$. Thus the elements $h_i$ vanish for $i\notin I$, and are unitaries for $i\in I$. We conclude that we have $X=\widehat{\Lambda}$, where $\Lambda=<h_i|i\in I>$ is the group generated by these unitaries.

In order to finish the proof, our claim is that for indices $i_x\in I$ we have:
$$g_{i_1}^{e_1}\ldots g_{i_p}^{e_p}=1\iff h_{i_1}^{e_1}\ldots h_{i_p}^{e_p}=1$$

Indeed, $\implies$ comes from the ergodicity condition, as processed above, and $\Longleftarrow$ comes from the existence of the morphism $\alpha$, which is given by $\alpha(h_i)=g_i$, for $i\in I$.
\end{proof}

Let us go back now to the general case, and discuss a number of further axiomatization issues, based on the examples that we have. We will need:

\begin{proposition}
The closed subspace $C_N^{I+}\subset U_N^+$ defined via
$$C(C_N^{I+})=C(U_N^+)\Big/\left<u\xi_I=\xi_I\right>$$
where $\xi_I=\frac{1}{\sqrt{|I|}}(\delta_{i\in I})_i$, is a compact quantum group.
\end{proposition}

\begin{proof}
We must check Woronowicz's axioms, and the proof goes as follows:

(1) Let us set $U_{ij}=\sum_ku_{ik}\otimes u_{kj}$. We have then:
\begin{eqnarray*}
(U\xi_I)_i
&=&\frac{1}{\sqrt{|I|}}\sum_{j\in I}U_{ij}
=\frac{1}{\sqrt{|I|}}\sum_{j\in I}\sum_ku_{ik}\otimes u_{kj}
=\sum_ku_{ik}\otimes(u\xi_I)_k\\
&=&\sum_ku_{ik}\otimes(\xi_I)_k
=\frac{1}{\sqrt{|I|}}\sum_{k\in I}u_{ik}\otimes1=(u\xi_I)_i\otimes1=(\xi_I)_i\otimes1
\end{eqnarray*}

Thus we can define indeed a comultiplication map, by $\Delta(u_{ij})=U_{ij}$.

(2) In order to construct the counit map, $\varepsilon(u_{ij})=\delta_{ij}$, we must prove that the identity matrix $1=(\delta_{ij})_{ij}$ satisfies $1\xi_I=\xi_I$. But this is clear.

(3) In order to construct the antipode, $S(u_{ij})=u_{ji}^*$, we must prove that the adjoint matrix $u^*=(u_{ji}^*)_{ij}$ satisfies $u^*\xi_I=\xi_I$. But this is clear from $u\xi_I=\xi_I$.
\end{proof}

Based on the computations that we have so far, we can formulate:

\begin{theorem}
Given a closed quantum subgroup $G\subset U_N^+$ and a set $I\subset\{1,\ldots,N\}$, we have a quotient map and an inclusion map as follows:
$$G/(G\cap C_N^{I+})\to X_{G,I}^{min}\subset X_{G,I}^{max}$$
These maps are both isomorphisms in the classical case. In general, they are both proper.
\end{theorem}

\begin{proof}
Consider the quantum group $H=G\cap C_N^{I+}$, which is by definition such that at the level of the corresponding algebras, we have $C(H)=C(G)\Big/\left<u\xi_I=\xi_I\right>$.

In order to construct a quotient map $G/H\to X_{G,I}^{min}$, we must check that the defining relations for $C(G/H)$ hold for the standard generators $x_i\in C(X_{G,I}^{min})$. But if we denote by $\rho:C(G)\to C(H)$ the quotient map, then we have, as desired:
$$(id\otimes\rho)\Delta x_i
=(id\otimes\rho)\left(\frac{1}{\sqrt{|I|}}\sum_{j\in I}\sum_ku_{ik}\otimes u_{kj}\right)
=\sum_ku_{ik}\otimes(\xi_I)_k=x_i\otimes1$$

In the classical case, Proposition 5.2 shows that both the maps in the statement are isomorphisms. For the group duals, however, these maps are not isomorphisms, in general. This follows indeed from Proposition 5.3, and from the general theory in \cite{bss}.
\end{proof}

It is quite unclear when the maps in Theorem 5.5 are both isomorphisms. Our conjecture is that this should happen when the dual of $G\subset U_N^+$ is amenable.

\section{Further examples}

We discuss now a number of further examples of affine homogeneous spaces, namely the quantum groups themselves, and their ``column spaces'' from \cite{bss}. We will need:

\begin{proposition}
Given a compact matrix quantum group $G=(G,u)$, the pair $G^t=(G,u^t)$, where $(u^t)_{ij}=u_{ji}$, is a compact matrix quantum group as well.
\end{proposition}

\begin{proof}
The construction of the comultiplication is as follows, where $\Sigma$ is the flip map:
$$\Delta^t[(u^t)_{ij}]=\sum_k(u^t)_{ik}\otimes(u^t)_{kj}\iff\Delta^t(u_{ji})=\sum_ku_{ki}\otimes u_{jk}\iff\Delta^t=\Sigma\Delta$$

As for the corresponding counit and antipode, these can be simply taken to be $(\varepsilon,S)$, and the conditions in Definition 1.1 above are satisfied.
\end{proof}

We will need as well the following result, which is standard as well:

\begin{proposition}
Given two closed subgroups $G\subset U_N^+$ and $H\subset U_M^+$, with fundamental corepresentations denoted $u=(u_{ij})$ and $v=(v_{ab})$, their product is a closed subgroup $G\times H\subset U_{NM}^+$, with fundamental corepresentation $w_{ia,jb}=u_{ij}\otimes v_{ab}$. 
\end{proposition}

\begin{proof}
The corresponding structural maps are $\Delta(\alpha\otimes\beta)=\Delta(\alpha)_{13}\Delta(\beta)_{24}$, $\varepsilon(\alpha\otimes\beta)=\varepsilon(\alpha)\varepsilon(\beta)$ and $S(\alpha\otimes\beta)=S(\alpha)S(\beta)$, the verifications being as follows:
\begin{eqnarray*}
\Delta(w_{ia,jb})&=&\Delta(u_{ij})_{13}\Delta(v_{ab})_{24}
=\sum_{kc}u_{ik}\otimes v_{ac}\otimes u_{kj}\otimes v_{cb}
=\sum_{kc}w_{ia,kc}\otimes w_{kc,jb}\\
\varepsilon(w_{ia,jb})&=&\varepsilon(u_{ij})\varepsilon(v_{ab})=\delta_{ij}\delta_{ab}=\delta_{ia,jb}\\
S(w_{ia,jb})&=&S(u_{ij})S(v_{ab})=v_{ba}^*u_{ji}^*=(u_{ji}v_{ba})^*=w_{jb,ia}^*
\end{eqnarray*}

We refer to Wang's paper \cite{wan} for more details regarding this construction.
\end{proof}

Let us call a closed quantum subgroup $G\subset U_N^+$ self-transpose when we have an automorphism $T:C(G)\to C(G)$ given by $T(u_{ij})=u_{ji}$. Observe that in the classical case, this amounts in $G\subset U_N$ to be closed under the transposition operation $g\to g^t$.

Finally, let us call $G\subset U_N^+$ reduced when its Haar functional is faithful. See \cite{wo1}.

With these notions in hand, let us go back to the affine homogeneous spaces. As a first result here, any closed subgroup $G\subset U_N^+$ appears as an affine homogeneous space over an appropriate quantum group, as follows:

\begin{proposition}
Given a reduced quantum subgroup $G\subset U_N^+$, we have an identification $X_{\mathcal G,I}^{min}\simeq G$, given at the level of standard coordinates by $x_{ij}=\frac{1}{\sqrt{N}}u_{ij}$, where:
\begin{enumerate}
\item $\mathcal G=G\times G^t\subset U_{N^2}^+$, with coordinates $w_{ia,jb}=u_{ij}\otimes u_{ba}$.

\item $I\subset\{1,\ldots,N\}^2$ is the diagonal set, $I=\{(k,k)|k=1,\ldots,N\}$.
\end{enumerate}
In the self-transpose case we can choose as well $\mathcal G=G\times G$, with $w_{ia,jb}=u_{ij}\otimes u_{ab}$.
\end{proposition}

\begin{proof}
In order to prove the first assertion, observe that $\alpha=\Delta$ and $\Phi=(id\otimes\Sigma)\Delta^{(2)}$ are given by the usual formulae for the affine homogeneous spaces, namely:
\begin{eqnarray*}
\alpha(u_{ij})&=&\sum_ku_{ik}\otimes u_{kj}=\sum_kw_{ij,kk}\\
\Phi(u_{ij})&=&\sum_{kl}u_{ik}\otimes u_{lj}\otimes u_{kl}=\sum_{kl}w_{ij,kl}\otimes u_{kl}
\end{eqnarray*}

The ergodicity condition being clear as well, this gives the result. 

Regarding now the last assertion, assume that we are in the self-transpose case, and so that we have an automorphism $T:C(G)\to C(G)$ given by $T(u_{ij})=u_{ji}$. The maps $\alpha=(id\otimes T)\Delta$ and $\Phi=(id\otimes T\otimes id)(id\otimes\Sigma)\Delta^{(2)}$ are then given by:
\begin{eqnarray*}
\alpha(u_{ij})&=&\sum_ku_{ik}\otimes u_{jk}=\sum_kw_{ij,kk}\\
\Phi(u_{ij})&=&\sum_{kl}u_{ik}\otimes u_{jl}\otimes u_{kl}=\sum_{kl}w_{ij,kl}\otimes u_{kl}
\end{eqnarray*}

Once again the ergodicity condition being clear as well, this gives the result.
\end{proof}

Let us discuss now the generalization of the above result, to the context of the spaces introduced in \cite{bss}. We recall from there that we have the following construction:

\begin{definition}
Given a closed subgroup $G\subset U_N^+$ and an integer $M\leq N$ we set
$$C(G_{N\times M})=\left<u_{ij}\Big|i\in\{1,\ldots,N\},j\in\{1,\ldots,M\}\right>\subset C(G)$$
and we call column space of $G$ the underlying quotient space $G\to G_{N\times M}$.
\end{definition}

As a basic example here, at $M=N$ we obtain $G$ itself. Also, at $M=1$ we obtain the space whose coordinates are those on the first column of coordinates on $G$. See \cite{bss}.

Given $G\subset U_N^+$ and an integer $M\leq N$, we can consider the quantum group $H=G\cap U_M^+$, with the intersection taken inside $U_N^+$, and with $U_M^+\subset U_N^+$ given by $u=diag(v,1_{N-M})$. Observe that we have a quotient map $C(G)\to C(H)$, given by $u_{ij}\to v_{ij}$. 

We have the following extension of Proposition 6.3:

\begin{theorem}
Given a reduced quantum subgroup $G\subset U_N^+$, we have an identification $X_{\mathcal G,I}^{min}\simeq G_{N\times M}$, given at the level of standard coordinates by $x_{ij}=\frac{1}{\sqrt{M}}u_{ij}$, where:
\begin{enumerate}
\item $\mathcal G=G\times H^t\subset U_{NM}^+$, where $H=G\cap U_M^+$, with coordinates $w_{ia,jb}=u_{ij}\otimes v_{ba}$.

\item $I\subset\{1,\ldots,N\}\times\{1,\ldots,M\}$ is the diagonal set, $I=\{(k,k)|k=1,\ldots,M\}$.
\end{enumerate}
In the self-transpose case we can choose as well $\mathcal G=G\times G$, with $w_{ia,jb}=u_{ij}\otimes v_{ab}$.
\end{theorem}

\begin{proof}
We will prove that the space $X=G_{N\times M}$, with coordinates $x_{ij}=\frac{1}{\sqrt{M}}u_{ij}$, coincides with the space $X_{\mathcal G,I}^{min}$ constructed in the statement, with its standard coordinates.

For this purpose, consider the following composition of morphisms, where in the middle we have the comultiplication, and at left and right we have the canonical maps:
$$C(X)\subset C(G)\to C(G)\otimes C(G)\to C(G)\otimes C(H)$$

The standard coordinates are then mapped as follows:
$$x_{ij}=\frac{1}{\sqrt{M}}u_{ij}\to\frac{1}{\sqrt{M}}\sum_ku_{ik}\otimes u_{kj}\to\frac{1}{\sqrt{M}}\sum_{k\leq M}u_{ik}\otimes v_{kj}=\frac{1}{\sqrt{M}}\sum_{k\leq M}w_{ij,kk}$$

Thus we obtain the standard coordinates on the space $X_{\mathcal G,I}^{min}$, as claimed. Finally, the last assertion is standard as well, by suitably modifying the above morphism.
\end{proof}

Let us mention that, with a little more work, one can prove that the spaces $G_{N\times M}^L$ from \cite{ba2}, depending on an extra parameter $L\in\{1,\ldots,M\}$, are covered as well by our formalism, the idea here being to truncate the index set, $I=\{(k,k)|k=1,\ldots,L\}$.

\section{The easy case}

We discuss now what happens when $G$ is easy, or more generally, motivated by the examples in section 6 above, when it is a product of easy quantum groups. 

Regarding easiness in general, we refer to \cite{bsp}, \cite{rwe}, \cite{twe}. In the context of the present paper,  let us go back to the Schur-Weyl considerations in section 4:
\begin{enumerate}
\item We would need there explicit bases $\{\xi_\pi|\pi\in D(k)\}$ for the spaces $Fix(u^{\otimes k})$, along with, if possible, explicit formulae for the vector entries $(\xi_\pi)_{i_1\ldots i_k}$. 

\item Equivalently, we would need bases $\{T_\pi|\pi\in D(k,l)\}$ for the spaces $Hom(u^{\otimes k},u^{\otimes l})$, along with explicit formulae for the matrix entries $(T_\pi)_{i_1\ldots i_k,j_1,\ldots j_l}$.
\end{enumerate} 

Here the equivalence between (1) and (2) is standard, see \cite{wo1}. Now in order to do so, one idea is to use set-theoretic partitions, and the following construction:

\begin{definition}
Associated to any partition $\pi\in P(k,l)$ is the linear map
$$T_\pi(e_{i_1}\otimes\ldots\otimes e_{i_k})=\sum_{j_1\ldots j_l}\delta_\pi\binom{i_1\ldots i_k}{j_1 \ldots j_l}e_{j_1}\otimes\ldots\otimes e_{j_l}$$
where $\delta_\pi\in\{0,1\}$ equals $1$ when the indices fit, and equals $0$ otherwise.
\end{definition}

Here $\pi\in P(k,l)$ means that $\pi$ has $k$ upper legs and $l$ lower legs, and by ``fitting'' we mean that, when putting the indices on the legs, each block contains equal indices.

In order to get now back to the quantum groups, we use Tannakian duality. Let us recall from \cite{bsp}, \cite{twe} that a category of partitions is a collection of subsets $D(k,l)\subset P(k,l)$, one for each choice of colored integers $k,l$, which is stable under vertical and horizontal concatenation, and under upside-down turning. With this convention, we have:

\begin{definition}
A closed quantum subgroup $G\subset U_N^+$ is called easy when we have
$$Hom(u^{\otimes k},u^{\otimes l})=span\left(T_\pi\Big|\pi\in D(k,l)\right)$$
for a certain category of partitions $D=(D(k,l))$.
\end{definition}

As basic examples, we have the groups $S_N,O_N,U_N$, coming from the categories of all partitions/pairings/matching pairings, and their free analogues $S_N^+,O_N^+,U_N^+$, coming from the categories of noncrossing partitions/pairings/matching pairings. See \cite{bsp}, \cite{twe}.

Now back to our homogeneous space questions, we have:

\begin{proposition}
When $G\subset U_N^+$ is easy, coming from a category of partitions $D$, the space $X_{G,I}\subset S^{N-1}_{\mathbb C,+}$ appears by imposing the relations
$$\sum_{i_1\ldots i_k}\delta_\pi(i_1\ldots i_k)x_{i_1}^{e_1}\ldots x_{i_k}^{e_k}=|I|^{|\pi|-k/2},\quad\forall k,\forall\pi\in D(k)$$
where $D(k)=D(0,k)$, and where $|.|$ denotes the number of blocks.
\end{proposition}

\begin{proof}
We know by easiness that $Fix(u^{\otimes k})$ is spanned by the vectors $\xi_\pi=T_\pi$, with $\pi\in D(k)$. According to Definition 7.1, these latter vectors are given by:
$$\xi_\pi=\sum_{i_1\ldots i_k}\delta_\pi(i_1\ldots i_k)e_{i_1}\otimes\ldots\otimes e_{i_k}$$

By applying now Theorem 4.3, with this particular choice of the vectors $\{\xi_\pi\}$, we deduce that $X_{G,I}\subset S^{N-1}_{\mathbb C,+}$ appears by imposing the following relations:
$$\sum_{i_1\ldots i_k}\delta_\pi(i_1\ldots i_k)x_{i_1}^{e_1}\ldots x_{i_k}^{e_k}=\frac{1}{\sqrt{|I|^k}}\sum_{b_1\ldots b_k\in I}\delta_\pi(b_1\ldots b_k),\quad\forall k,\forall\pi\in D(k)$$

Now since the sum on the right equals $|I|^{|\pi|}$, this gives the result.
\end{proof}

More generally now, in view of the examples from section 6 above, making the link with \cite{bss}, it is interesting to work out what happens when $G$ is a product of easy quantum groups, and the index set $I$ appears as $I=\{(c,\ldots,c)|c\in J\}$, for a certain set $J$.

The result here, in its most general form, is as follows:

\begin{theorem}
For a product of easy quantum groups, $G=G_{N_1}^{(1)}\times\ldots\times G_{N_s}^{(s)}$, and with $I=\{(c,\ldots,c)|c\in J\}$, the space $X_{G,I}\subset S^{N-1}_{\mathbb C,+}$ appears by imposing the relations
$$\sum_{i_1\ldots i_k}\delta_\pi(i_1\ldots i_k)x_{i_1}^{e_1}\ldots x_{i_k}^{e_k}=|J|^{|\pi_1\vee\ldots\vee\pi_s|-k/2},\quad\forall k,\forall\pi\in D^{(1)}(k)\times\ldots\times D^{(s)}(k)$$
where $D^{(r)}\subset P$ is the category of partitions associated to $G_{N_r}^{(r)}\subset U_{N_r}^+$, and where the partition $\pi_1\vee\ldots\vee\pi_s\in P(k)$ is the one obtained by superposing $\pi_1,\ldots,\pi_s$.
\end{theorem}

\begin{proof}
Since we are in a direct product situation, $G=G_{N_1}^{(1)}\times\ldots\times G_{N_s}^{(s)}$, the general theory in \cite{wan} applies, and shows that a basis for $Fix(u^{\otimes k})$ is provided by the vectors $\rho_\pi=\xi_{\pi_1}\otimes\ldots\otimes\xi_{\pi_s}$, with $\pi=(\pi_1,\ldots,\pi_s)\in D^{(1)}(k)\times\ldots\times D^{(s)}(k)$.

Once again Theorem 4.3 applies, and shows that the space $X_{G,I}\subset S^{N-1}_{\mathbb C,+}$ appears by imposing the following relations to the standard coordinates:
$$\sum_{i_1\ldots i_k}\delta_\pi(i_1\ldots i_k)x_{i_1}^{e_1}\ldots x_{i_k}^{e_k}=\frac{1}{\sqrt{|I|^k}}\sum_{b_1\ldots b_k\in I}\delta_\pi(b_1\ldots b_k),\quad\forall k,\forall\pi\in D^{(1)}(k)\times\ldots\times D^{(s)}(k)$$

Since the conditions $b_1,\ldots,b_k\in I$ read $b_1=(c_1,\ldots,c_1),\ldots,b_k=(c_k,\ldots,c_k)$, for certain elements $c_1,\ldots c_k\in J$, the sums on the right are given by:
\begin{eqnarray*}
\sum_{b_1\ldots b_k\in I}\delta_\pi(b_1\ldots b_k)
&=&\sum_{c_1\ldots c_k\in J}\delta_\pi(c_1,\ldots,c_1,\ldots\ldots,c_k,\ldots,c_k)\\
&=&\sum_{c_1\ldots c_k\in J}\delta_{\pi_1}(c_1\ldots c_k)\ldots\delta_{\pi_s}(c_1\ldots c_k)\\
&=&\sum_{c_1\ldots c_k\in J}\delta_{\pi_1\vee\ldots\vee\pi_s}(c_1\ldots c_k)
\end{eqnarray*}

Now since the sum on the right equals $|J|^{|\pi_1\vee\ldots\vee\pi_s|}$, this gives the result.
\end{proof}

\section{Probabilistic aspects}

Consider the spaces $X=X_{G,I}$ from Theorem 7.4. Our purpose now will be to establish some liberation results, in the sense of the Bercovici-Pata bijection \cite{bpa}. 

As in \cite{ba1}, \cite{ba2}, we use suitable sums of ``non-overlapping'' coordinates. To be more precise, since we are in a direct product situation, in $N=N_1\ldots N_s$ dimensions, we can consider ``diagonal'' coordinates $x_{i\ldots i}$, and then sum them over various indices $i$.

As a first result regarding such variables, we have:

\begin{proposition}
The moments of the variable $\chi_T=\sum_{i\leq T}x_{i\ldots i}$ are given by
$$\int_X\chi_T^k\simeq\frac{1}{\sqrt{M^k}}\sum_{\pi\in D^{(1)}(k)\cap\ldots\cap D^{(s)}(k)}\left(\frac{TM}{N}\right)^{|\pi|}$$
in the $N_i\to\infty$ limit, $\forall i$, where $M=|I|$, and $N=N_1\ldots N_s$.
\end{proposition}

\begin{proof}
We have the following formula:
$$\pi(x_{i_1\ldots i_s})=\frac{1}{\sqrt{M}}\sum_{c\in J}u_{i_1c}\otimes\ldots\otimes u_{i_sc}$$

For the variable in the statement, we therefore obtain:
$$\pi(\chi_T)=\frac{1}{\sqrt{M}}\sum_{i\leq T}\sum_{c\in J}u_{ic}\otimes\ldots\otimes u_{ic}$$

Now by raising to the power $k$ and integrating, we obtain:
\begin{eqnarray*}
\int_X\chi_T^k
&=&\frac{1}{\sqrt{M^k}}\sum_{i_1\ldots i_k\leq T}\sum_{c_1\ldots c_k\in J}\int_{G^{(1)}}u_{i_1c_1}\ldots u_{i_kc_k}\ldots\ldots\int_{G^{(s)}}u_{i_1c_1}\ldots u_{i_kc_k}\\
&=&\frac{1}{\sqrt{M^k}}\sum_{ic}\sum_{\pi\sigma}\delta_{\pi_1}(i)\delta_{\sigma_1}(c)W_{kN_1}^{(1)}(\pi_1,\sigma_1)\ldots\delta_{\pi_s}(i)\delta_{\sigma_s}(c)W_{kN_s}^{(s)}(\pi_s,\sigma_s)\\
&=&\frac{1}{\sqrt{M^k}}\sum_{\pi\sigma}T^{|\pi_1\vee\ldots\vee\pi_s|}M^{|\sigma_1\vee\ldots\vee\sigma_s|}
W_{kN_1}^{(1)}(\pi_1,\sigma_1)\ldots W_{kN_s}^{(s)}(\pi_s,\sigma_s)
\end{eqnarray*}

We use now the standard fact, from \cite{ba2}, that the Weingarten functions are concentrated on the diagonal. Thus in the limit we must have $\pi_i=\sigma_i$ for any $i$, and we obtain:
\begin{eqnarray*}
\int_X\chi_T^k
&\simeq&\frac{1}{\sqrt{M^k}}\sum_\pi T^{|\pi_1\vee\ldots\vee\pi_s|}M^{|\pi_1\vee\ldots\vee\pi_s|}N_1^{-|\pi_1|}\ldots N_s^{-|\pi_s|}\\
&\simeq&\frac{1}{\sqrt{M^k}}\sum_{\pi\in D^{(1)}\cap\ldots\cap D^{(s)}}T^{|\pi|}M^{|\pi|}(N_1\ldots N_s)^{-|\pi|}\\
&=&\frac{1}{\sqrt{M^k}}\sum_{\pi\in D^{(1)}\cap\ldots\cap D^{(s)}}\left(\frac{TM}{N}\right)^{|\pi|}
\end{eqnarray*}

But this gives the formula in the statement, and we are done.
\end{proof}

As a consequence, we have the following result:

\begin{theorem}
In the context of a liberation operation for quantum groups, $G^{(i)}\to G^{(i)+}$, the laws of the variables $\sqrt{M}\chi_T$ are in Bercovici-Pata bijection, in the $N_i\to\infty$ limit.
\end{theorem}

\begin{proof}
Assume indeed that we have easy quantum groups $G^{(1)},\ldots,G^{(s)}$, with free versions $G^{(1)+},\ldots,G^{(s)+}$. At the level of the categories of partitions, we have:
$$\bigcap_i\left(D^{(i)}\cap NC\right)=\left(\bigcap_iD^{(i)}\right)\cap NC$$

Since the intersection of Hom-spaces is the Hom-space for the generated quantum group, we deduce that at the quantum group level, we have:
$$<G^{(1)+},\ldots,G^{(s)+}>=<G^{(1)},\ldots,G^{(s)}>^+$$

Thus the result follows from Proposition 8.1, and from the Bercovici-Pata bijection result for truncated characters for this latter liberation operation \cite{bsp}, \cite{twe}.
\end{proof}

As a conclusion, Theorem 7.4 provides a quite reasonable definition for the notion of ``easy homogeneous space''. There are of course several potential extensions to be explored, by using for instance the more general notions from \cite{fre}, \cite{swe}. Interesting as well would be to try to understand what an ``easy algebraic manifold'' should be, independently of the quantum group context. Observe that this latter question makes indeed sense, because in the context of the general considerations in section 3 above, $G\subset U_N^+$ appears as a certain uniquely determined quantum subgroup of the affine quantum isometry group of $X\subset S^{N-1}_{\mathbb C,+}$. Thus, an axiomatization of the easy algebraic manifolds is in principle possible, without direct reference to the underlying compact quantum groups.


\begin{thebibliography}{99}

\bibitem{ba1}T. Banica, The algebraic structure of quantum partial isometries, {\em Infin. Dimens. Anal. Quantum Probab. Relat. Top.} {\bf 19} (2016), 1--36.

\bibitem{ba2}T. Banica, Liberation theory for noncommutative homogeneous spaces, {\em Ann. Fac. Sci. Toulouse Math.} {\bf 26} (2017), 127--156.

\bibitem{bco}T. Banica and B. Collins, Integration over compact quantum groups, {\em Publ. Res. Inst. Math. Sci.} {\bf 43} (2007), 277--302.

\bibitem{bgo}T. Banica and D. Goswami, Quantum isometries and noncommutative spheres, {\em Comm. Math. Phys.} {\bf 298} (2010), 343--356.

\bibitem{bss}T. Banica, A. Skalski and P.M. So\l tan, Noncommutative homogeneous spaces: the matrix case, {\em J. Geom. Phys.} {\bf 62} (2012), 1451--1466.

\bibitem{bsp}T. Banica and R. Speicher, Liberation of orthogonal Lie groups, {\em Adv. Math.} {\bf 222} (2009), 1461--1501.

\bibitem{bpa}H. Bercovici and V. Pata, Stable laws and domains of attraction in free probability theory, {\em Ann. of Math.} {\bf 149} (1999), 1023--1060.

\bibitem{boc}F. Boca, Ergodic actions of compact matrix pseudogroups on C$^*$-algebras, {\em Ast\'erisque} {\bf 232} (1995), 93--109.

\bibitem{csn}B. Collins and P. \'Sniady, Integration with respect to the Haar measure on the unitary, orthogonal and symplectic group, {\em Comm. Math. Phys.} {\bf 264} (2006), 773--795.

\bibitem{dya}K. De Commer and M. Yamashita, Tannaka-Krein duality for compact quantum homogeneous spaces. I. General theory, {\em Theory Appl. Categ.} {\bf 28} (2013), 1099--1138. 

\bibitem{fre}A. Freslon, On the partition approach to Schur-Weyl duality and free quantum groups, preprint 2014.

\bibitem{kso}P. Kasprzak and P.M. So\l tan, Embeddable quantum homogeneous spaces, {\em J. Math. Anal. Appl.} {\bf 411} (2014), 574--591.

\bibitem{pod}P. Podle\'s, Symmetries of quantum spaces. Subgroups and quotient spaces of quantum SU(2) and SO(3) groups, {\em Comm. Math. Phys.} {\bf 170} (1995), 1--20.

\bibitem{rwe}S. Raum and M. Weber, The full classification of orthogonal easy quantum groups, {\em Comm. Math. Phys.} {\bf 341} (2016), 751--779.

\bibitem{swe}R. Speicher and M. Weber, Quantum groups with partial commutation relations, preprint 2016. 

\bibitem{twe}P. Tarrago and M. Weber, Unitary easy quantum groups: the free case and the group case, preprint 2015.

\bibitem{wan}S. Wang, Free products of compact quantum groups, {\em Comm. Math. Phys.} {\bf 167} (1995), 671--692.

\bibitem{wei}D. Weingarten, Asymptotic behavior of group integrals in the limit of infinite rank, {\em J. Math. Phys.} {\bf 19} (1978), 999--1001.

\bibitem{wo1}S.L. Woronowicz, Compact matrix pseudogroups, {\em Comm. Math. Phys.} {\bf 111} (1987), 613--665.

\bibitem{wo2}S.L. Woronowicz, Tannaka-Krein duality for compact matrix pseudogroups. Twisted SU(N) groups, {\em Invent. Math.} {\bf 93} (1988), 35--76.

\end{thebibliography}
\end{document}